\documentclass [11pt]{article} 
\usepackage{a4}                
\usepackage{paralist}          
\usepackage{graphicx}
\usepackage{amssymb}                              
\usepackage{float}
\usepackage{amsmath}
\usepackage{psfrag}
\usepackage{amsmath,amsthm}
\usepackage{cleveref}

\usepackage{amsmath,amsthm}

\theoremstyle{plain}
\newtheorem{thm}{Theorem}[section]

\newtheorem{lem}[thm]{Lemma}

\newtheorem{cor}[thm]{Corollary}

\theoremstyle{definition}
\newtheorem{defn}[thm]{Definition}
\newtheorem{ex}[thm]{Example}

\theoremstyle{remark}
\newtheorem{rem}[thm]{Remark}

\title{\textbf{Positive rank gradient and $p$-largeness for groups defined by presentations with $p$-deficiency less than or equal to one}}
\author{Mariano Zer\'{o}n-Medina Laris}

\begin{document}

\maketitle

\begin{abstract}
It is an immediate consequence of the results in \cite{puchta-yiftach} that a presentation with $p$-deficiency greater than one defines a group with positive rank gradient. By results in \cite{jack-anitha}, we know that a finite presentation with $p$-deficiency greater than one defines a $p$-large group. In both \cite{puchta-yiftach} and \cite{anitha}, extensions of these results were obtained for the $p$-deficiency one case. In this paper we consider the case when the presentation has $p$-deficiency less than or equal to one.
\end{abstract}

\section{Introduction}

Say $G$ is a finitely generated group given by a presentation $Q=\langle X|R\rangle$, where $X$ freely generates $F_n$, the non-abelian free group of rank $n$.  The author of \cite{puchta} defined the $p$-deficiency of $Q$ as follows. Let $\nu_{p}(r)$, where $r\in R$, be the largest integer $k$, such that there is some $s\in F_n$ with $r=s^{p^{k}}$. The \emph{$p$-deficiency} of $Q$ is defined as 
\[
def_p(Q)=n-\sum_{r\in R}\ p^{-\nu_{p}(r)}-1,
\]
and the $p$-deficiency $def_p(G)$ of $G$, as the supremum of the deficiencies of all presentations of $G$ with finite generating set. In \cite{jack-anitha}, the $p$-deficiency of $Q$ is defined as
\[
n-\sum_{r\in R}\ p^{-\nu_{p}(r)}.
\]
Throughout this paper, we make use of the latter definition.

\vskip 2mm

The $p$-deficiency of a group $G$ is related to invariants such as largeness and the rank gradient of a group. The \emph{rank gradient} of a finitely generated group $G$ is defined by
\[
RG(G):= \underset{H\underset{f}{\leqslant}G }{\text{inf}} \left\{ \dfrac{d(H)-1}{|G:H|} \right\},
\]
where $H$ is a finite index subgroup of $G$ and $d(H)$ the rank of $H$. Having positive rank gradient (that is, $RG(G)>0$) is a strong property which is invariant under finite index subgroups and finite index supergroups. By the Nielson-Schreier index formula $RG(F_n)=n-1$. In \cite{lackers RG}, M. Lackenby proved that if $G$ is a non-trivial free product $A\ast B$, where either $A$ or $B$ is not isomorphic to $C_2$, the cyclic group of order two, then $RG(G)>0$. On the other hand, $SL_n(\mathbb{Z})$ for $n\geq 3$, ascending $HNN$-extensions, and direct products of finitely generated infinite residually finite groups all have zero rank gradient (\cite{lackers RG}, \cite{abert-nikolov}).
\vskip 2mm

It is worth pointing, as we will use this fact later, that the rank minus one is \emph{submultiplicative} with respect to finite index subgroups, which means
\begin{equation}\label{eq; submult rank}
\dfrac{d(H)-1}{|G:H|}\leq d(G)-1,
\end{equation}
whenever $H$ is a finite index subgroup of $G$. 
\vskip 2mm

Theorem $3$ in \cite{puchta-yiftach} gives the necessary tool to show that if $G$ has $p$-deficiency greater than one, then the group $G$ has positive rank gradient (see \cref{rem: pos p def implies pos rank grad}). Similar ideas lead to the construction of finitely generated infinite residually finite $p$-groups with positive rank gradient (\cite{puchta}). Apart from the one in \cite{osin}, this is the only other construction of such groups.
\vskip 2mm

A group $G$ is large if it has a finite index normal subgroup $H$ with a non-abelian free quotient. A group $G$ is $p$-large if it has a normal subgroup of index in $F_n$ a power of a prime $p$ which has a non-abelian free quotient. Clearly, a $p$-large group is large. However, the converse does not hold. For instance, in \cite{jack-anitha}, the authors show that $A_5\ast A_5$ is not $p$-large for all primes $p$, but the group is nonetheless large. We refer to \cite{jack-anitha} for a more thorough treatment of $p$-largeness, and to \cite{lackers detecting large groups} for a characterisation of it.
\vskip 2mm

Largeness is a strong property invariant under finite index subgroups and finite supergroups, as well as prequotients. Moreover, if a group $G$ is large, then $G$ contains a non-abelian free subgroup, $G$ is $SQ$-universal, $G$ has finite index subgroups with arbitrarily large Betti number, among other properties.
\vskip 2mm

Examples of large groups include all those groups defined by presentations with at least two more generators than relators (such as non-abelian free groups); fundamental groups of closed orientable surfaces of genus greater than one; fundamental groups of closed $3$-dimensional hyperbolic manifolds; free products $G_1\ast G_2$ where both $G_1$ and $G_2$ have proper finite index subgroups and at least one of them has a finite index subgroup of index at least $3$; some families of mapping tori (\cite{button}); some families of triangle groups (\cite{gen traingle groups}); among other groups.
\vskip 2mm 

The authors of \cite{jack-anitha}, proved that if $G$ has $p$-deficiency greater than one, then $G$ is $p$-large.
\vskip 2mm

If a presentation has $p$-deficiency one, then neither largeness nor positive rank gradient may be concluded. For instance, consider the integers with its usual presentation or the infinite dihedral group $D_{\infty} = \langle x_1,x_2 | x_1^{2}, x_2^{2}\rangle$. The first has $p$-deficiency one for every prime $p$, while the second has $2$-deficiency one. However, neither are large nor have positive rank gradient; all the finite index subgroups of $\mathbb{Z}$ are isomorphic to $\mathbb{Z}$, while $D_{\infty}$ has a copy of $\mathbb{Z}$ as a finite index subgroup. Nevertheless, under suitable conditions, a $p$-deficiency one presentation defines a group with strong properties. For example, in \cite{puchta-yiftach}, the authors found conditions for when a presentation with $p$-deficiency one has a finite index subgroup with $p$-deficiency greater than one. As both largeness and positive rank gradient are invariant under finite supergroups, the whole group enjoys these properties too. Also, it was proved in \cite{anitha} that if the presentation is finite and has $p$-deficiency one, then the group it defines has a finite index subgroup with positive first Betti number.
\vskip 2mm

There are other properties worth mentioning in relation to largeness and having positive rank gradient. A finitely generated group $G$ has \emph{property} ($T$) if every isometric action of $G$ on a Hilbert space has a global fixed point. It has \emph{property} ($\tau$) if for some (equivalently any) finite generating set $S$ for $G$, the set of Cayley graphs Cay$(G/N,S)$ forms an expander family, where $N$ varies over all finite index normal subgroups of $G$. It is amenable if it admits a finitely additive, left invariant, probability measure. Property ($T$) implies ($\tau$), but not vice-versa. However, an amenable group with ($T$) must be finite, as must a residually finite amenable group with ($\tau$). All three properties are preserved under quotients, extensions and subgroups of finite index, whereas amenability is further preserved under arbitrary subgroups.
\vskip 2mm

The integers are an example of a group which does not have property ($\tau$). Since having property ($\tau$) is invariant under finite index subgroups and quotients, any group with a finite index subgroup that surjects onto the integers, such as a large group, does not have property ($\tau$) and hence not property ($T$). Also, large groups contain non-abelian free subgroups. Since having non-abelian free subgroups implies non-amenability, then large groups are non-amenable too. Finally, if the group is residually finite, having positive rank gradient implies the group is non-amenable. From this perspective, properties ($T$), ($\tau$) and amenability, are antipodal to largeness and having positive rank gradient. Therefore, by proving the latter properties we exclude the former ones.
\vskip 2mm

In this paper we consider groups defined by presentations with $p$-deficiency less than or equal to one. By using ideas from \cite{puchta-yiftach} and \cite{anitha} we get a more complete picture of when these groups are large or have positive rank gradient.
\vskip 2mm

\textbf{Acknowledgements}
\vskip 2mm

This work is part of the author's PhD done under the supervision of Jack Button. The author would like to thank Jack for all the support and advice needed for the completion of this work. Also, the author would like to thank the Mexican Council of Science and Technology (CONACyT) and the Cambridge Overseas Trust for their financial support all these years.

\section{Presentations with $p$-deficiency less than or equal to one}

From now on, we work with finitely generated groups. Whenever $P=\langle X|R\rangle$ is a presentation for the finitely generated group $G$, then $X$ freely generates a non-abelian free group of finite rank. Denote the canonical map induced by $P$ from $F$, the non-abelian free group of finite rank freely generated by $X$, to $G$, by $\varphi$. Call the intersection of finite index subgroups in $G$ the \emph{finite residual} and denote it by $R_G$. Call $G/R_G$ the \emph{residual quotient}. Denote the composition of $\varphi$ followed by the canonical map from $G$ to $G/R_G$ by $\psi$.
\vskip 2mm

Define the $p$-rank gradient as
\[
RG_p(G):= \underset{H\underset{f}{\trianglelefteq}G }{\text{inf}} \left\{ \dfrac{d_p(H)-1}{|G:H|} \right\},
\]
\begin{lem}\label{lemma pdef>1 implies RG_p>0}
Let $G$ be a finitely generated group with $p$-deficiency greater than one. Then $RG_p(G)>0$.
\end{lem}

\begin{proof}
In \cite{puchta-yiftach} it was proved that $p$-deficiency minus one is supermultiplicative with respect to finite index normal subgroups of $G$. That is, if $H$ is a finite index normal subgroup of $G$, then
\[
def_p(G)-1\leq \dfrac{def_p(H)-1}{|G:H|}.
\]
Moreover, the $p$-deficiency of a group is a lower bound for the $p$-rank of a group (\cite{jack-anitha}, ($2.1$)), hence
\begin{equation}\label{eq: in first lemma}
def_p(G)-1\leq \dfrac{def_p(H)-1}{|G:H|}\leq \dfrac{d_p(H)-1}{|G:H|}.
\end{equation}
The result is obtained by taking the infimum over all the finite index normal subgroups of $G$.
\end{proof}
\vskip 2mm

\begin{rem}\label{rem: pos p def implies pos rank grad}
The $p$-rank of a group $G$ is smaller than or equal to the rank of $G$, hence from the proof of the previous lemma
\[
def_p(G)-1\leq\dfrac{d(H)-1}{|G:H|},
\]
where $H$ is a finite index normal subgroup of $G$. Taking the infimum over all finite index normal subgroups of $G$ on the right hand side gives
\[
def_p(G)-1\leq RG(G),
\]
and hence a group with $p$-deficiency greater than one has positive rank gradient.
\vskip 2mm
\end{rem}
\vskip 2mm

\begin{rem}
Let $G$ be a finitely generated group. Given $H$ a finite index subgroup of $G$, consider $Core_G(H)$ the core of $H$ in $G$. The core of $H$ in $G$  is defined as the intersection of all conjugates of $H$ by elements in $G$. The core of $H$ is a finite index normal subgroup of $G$ and hence by (\cref{eq; submult rank})
\[
\dfrac{d(Core_G(H))-1}{|G:Core_G(H)|}\leq\dfrac{d(H)-1}{|G:H|}\leq d(G)-1.
\]
Therefore, the rank gradient of $G$ is also computed by only considering the set of finite index normal subgroups in $G$.
\end{rem}
\vskip 2mm

In \cref{thm: main of section} (below) we consider a finitely generated group $G$ with a presentation $Q=\langle X|R\rangle$ which has the following characteristics. The set $X$ is finite and freely generates the non-abelian free group of rank $n$. The set $R$ may be separated into three disjoint sets. The first consists of a finite collection of elements $w_1,\ldots,w_r$ in $F_n$ which are not expressed as proper $p$-powers. Denote this set by $S_1$.
\vskip 2mm

The second set, which may be infinite, consists of elements $v^{p^{a}}$ in $F_n$, where $a>0$, such that the order of $\psi(v)$ in $G/R_G$, which we denote by $o(\psi(v),G/R_G)$, is exactly $p^a$. Denote this set by $S_2$. 
\vskip 2mm

The third set, which may also be infinite, consists of elements $u^{p^b}$ with $b>0$, such that $o(\psi(u),G/R_G)<p^b$. Denote this set by $S_3$.
\vskip 2mm

Construct a presentation $P$ from $Q$ in the following way. Take $X$ to be the set of generators of $P$ and take $S=S_1\cup S_2$ to be the set of relators of $P$. 
\vskip 2mm

\begin{thm}\label{thm: main of section}
Let $G$ be a finitely generated group with a presentation $Q=\langle X|R\rangle$ as described above. Consider the presentation $P=\langle X|S\rangle$ obtained from $Q$. If $def_p(P)>1$, then $RG(G)>0$.
\end{thm}
\begin{proof}

Denote the normal closure of $R$ in $F_n$ by $N$ and the normal closure of $S$ in $F_n$ by $M$. Let $K$ be the group defined by the presentation $P$. By assumption $def_p(P)>1$, therefore \cref{lemma pdef>1 implies RG_p>0} implies $RG_p(K)>0$. If $N=M$, then $G\cong K$, and hence $RG(G)>0$. Therefore, assume they are not.

\vskip 2mm

Consider the set $A$ of normal subgroups of finite index in $F_n$ which contain $N$. Similarly, define the set $B$ to be the set of normal subgroups of finite index in $F_n$ which contain $M$. Note that $M\lhd N\trianglelefteq F_n$ and $A\subseteq B$. Moreover, the set $A$ is in bijection with the set of all finite index normal subgroups of $G$, while the set $B$ is in bijection with the set of all finite index normal subgroups of $K$. The key is to prove that for every element $H$ in $A$, $d_p(H/N)=d_p(H/M)$. This would then imply
\vskip 1mm

\[
RG_p(K)=\underset{H\in B}{\text{inf}} \left\{ \dfrac{d_p(H/M)-1}{|F_n:H|} \right\}\leq \underset{H\in A}{\text{inf}} \left\{ \dfrac{d_p(H/M)-1}{|F_n:H|} \right\}
\]
\vskip 1mm
\[
= \underset{H\in A}{\text{inf}} \left\{ \dfrac{d_p(H/N)-1}{|F_n:H|} \right\}\leq \underset{H\in A}{\text{inf}} \left\{ \dfrac{d(H/N)-1}{|F_n:H|} \right\}= RG(G),
\]
\\
from where the result follows.

\vskip 2mm

Take $H$ in $A$. Compute a presentation for $H/N$ and a presentation for $H/M$ by using the Reidemeister-Schreier rewriting process (page $103$, \cite{lyndon}). Both presentations will have a generating set $Y$ with $(n-1)|F_n:H|+1$ elements. For each element in $R$, $H/N$ will have $|F_n:H|$ relators. The same goes for $H/M$ and $S$. As $S$ is contained in $R$, then the relators in the presentations for $H/N$ and $H/M$ coming from the relators in $S$ will be the same.
\vskip 2mm

Take a relator $u^{p^b}$ in $R$ which does not belong to $S$, that is, one that belongs to $S_3$. As $o(\psi(u),G/R_G)<p^b$, then $u^{p^{b-1}}$ is in $H$ and hence $u^{p^{b}}$ can be written as a $p$-th power of an element in $H$. As $u^{p^{b}}$ can be written as the $p$-th power of an element in $H$, then the $p$-rank of $H/M$ is the same as the $p$-rank of $H/\langle\langle M\cup u^{p^{b}}\rangle\rangle$. As this holds for all elements in $S_3$, then the $p$-rank of $H/N$ and the $p$-rank of $H/M$ are the same.

\end{proof}
\vskip 2mm

\begin{cor}\label{cor: rg>0 when order of minimal p root is less than its p power}
Let $G$ be a finitely generated group with presentation $Q=\langle X|R\rangle$, such that $def_p(Q)=1$. Say $X$ is finite and $R=\{ w_1,\ldots,w_r,w_{r+1}^{p^{a_1}},\ldots, w_{r+j}^{p^{a_j}},\ldots \}$, where $j\geq 1$ and $a_j\geq 1$. Suppose the order of $\psi(w_{r+j})$ in $G/R_G$ is strictly less than $p^{a_j}$, for some $j\geq 1$. Then $RG(G)>0$.
\end{cor}
\begin{proof}
Consider the presentation $P=\langle X|S\rangle$, where $S$ consists of all the elements of $R$ except $w_{r+j}^{p^{a_j}}$. The result follows from the arguments in \cref{thm: main of section} noting that $def_p(P)>1$.
\end{proof}

\vskip 2mm

The following result says that if $Q$ is finite, the conditions in \cref{thm: main of section} imply $G$ is $p$-large. The proof of this result follows very similar arguments to the ones in Theorem $3$ of \cite{anitha}. However, the proof presented here is more general as it does not impose the condition that the presentation $Q$ must have $p$-deficiency equal to one.
\vskip 2mm

\begin{thm}\label{thm: main section p large}
Let $G$ be a finitely presented group with finite presentation $Q=\langle X|R\rangle$.  Let $P=\langle X| S\rangle $ be as in \cref{thm: main of section}. If def$_p(P)>1$, then $G$ is $p$-large.
\end{thm}
\begin{proof}
The group $G$ has the following presentation
\[
Q=\langle x_1,\ldots,x_n \mid w_1,\ldots,w_r,v_{i}^{p^{a_i}},u_{j}^{p^{b_j}}\rangle,
\]
where $i\in I$, $j\in J$, and $I$ and $J$ are finite collections of indices.
\vskip 2mm

Set $def_p(P)-1=\varepsilon$. By hypothesis $\varepsilon>0$. Consider large enough positive integers $b_j'$, such that $\sum_{j\in J}1/p^{b_j'}<\varepsilon$. In the same spirit as \cite{anitha}, let $G'$ be the group given by the presentation
\[
Q'=\langle X|R'\rangle=\langle x_1,\ldots,x_n \mid w_1,\ldots,w_r,v_{i}^{p^{a_i}},u_j^{p^{b_j'}}\rangle.
\]
\vskip 2mm

Denote the normal closures of $R$ and $R'$ in $F_n$, by $N$ and $N'$, respectively. Since $Q'$ is finitely presented and def$_p(Q')>1$, then $G'$ is $p$-large. This means there is a normal subgroup $H$ in $F_n$ with index a power of $p$, which contains $N'$, such that $H/N'$ surjects onto $F_2$.
\vskip 2mm

Say there are $l$ elements in $S_3$, $u_1^{p^{b_1}},\ldots,u_l^{p^{b_l}}$. We claim that if $u^{p^{b_i}}\in H$ for all $i$, $1\leq i\leq l$, then $H/N$ surjects onto $F_2$ and hence $G$ is $p$-large.
\vskip 2mm

Denote by $\phi$ the map from $H/N'$ to $F_2$. Let $ker(\varphi)$ be the kernel of $\phi$. Clearly, $N'\trianglelefteq ker(\varphi)\triangleleft H$. If $u^{p^{b_i}}\in H$ for some $i$, then $u^{p^{b_i}}\in ker(\varphi)$. Otherwise $\phi(u^{p^{b_i}})$ would be non-zero in $F_2$ and hence it would have infinite order in $F_2$, which implies $(u^{p^{b_i}})^{s}\notin ker(\varphi)$ for all $s\in \mathbb{Z}$. This is a contradiction as $u^{p^{b'_i}}\in N'$. As $u^{p^{b_i}}\in ker(\varphi)$, for all $i$, $1\leq i\leq l$, then $N\trianglelefteq ker(\varphi)\triangleleft H$, which means $H/N$ also surjects onto a non-abelian free group of finite rank.
\vskip 2mm

Now we prove that $u_i^{{p}^{b_i}}\in H$ for all $i$, $1\leq i\leq l$. Suppose $u_j^{p^{b_j}}\notin H$ for some $j$, $1\leq j\leq l$. We will show that this condition implies there is a finite index normal subgroup $K$ of $F_n$ that contains $N$, such that $o(u_k,F_n/K)=p^{b_k}$, for some $k$, $1\leq k\leq l$. This contradicts the assumption that $o(\psi(u_i),G/R_G)<p^{b_i}$, for all $i$, $1\leq i\leq l$.
\vskip 2mm

Let $K$ be the normal subgroup in $F_n$ generated by $\{u_1^{p^{b_1}},\ldots,u_l^{p^{b_l}}\}$ and $H$. As $u_j^{p^{b_j}}\notin H$, then $H$ is properly contained in $K$. Moreover, $u_i^{p^{b_i}}\in K$ for all $i$, $1\leq i\leq l$, hence $N$ is contained in $K$.
\vskip 2mm

As $F_n/H$ is a finite $p$-group, then $K/H$ is a finite $p$-group too. From now on, denote $K/H$ by $L$ and the image of $u_i^{p^{b_i}}$ in $L$ by $\overline{u}_i^{p^{b_i}}$, for all $i$, $1\leq i\leq l$. Consider the Frattini subgroup of $L$, $\mathfrak{F}(L)\triangleleft L$. As $L$ is a finite $p$-group, $\mathfrak{F}(L)$ is properly contained in $L$ and $\mathfrak{F}(L)=(L)^{p}[L,L]$, where $[L,L]$ denotes the commutator subgroup of $L$. Since $\mathfrak{F}(L)$ is a characteristic subgroup of $L$, then it is normal in $F_n/H$. Therefore, if $\overline{u}_i^{p^{b_i}}\in \mathfrak{F}(L)$ for all $i$, $1\leq i\leq l$, then $\mathfrak{F}(L)=L$, which is a contradiction. This means there is an element $\overline{u}_k^{p^{b_k}}$, where $1\leq k\leq l$, such that $\overline{u}_k^{p^{b_k}}\notin \mathfrak{F}(L)$.
\vskip 2mm

If $\overline{u}_k^{p^{b_k-1}}\in L$, then $(\overline{u}_k^{p^{b_k-1}})^{p}=\overline{u}_k^{p^{b_k}}$ would be in $\mathfrak{F}(L)$. As the latter cannot hold by assumption, then $\overline{u}_k^{p^{b_k-1}}\notin L$. Therefore $o(u_k,F_n/K)=p^{b_k}$, which implies $o(\psi(u_k),G/R_G)=p^{b_k}$.

\end{proof}

\vskip 2mm

\begin{cor}$(\cite{anitha}$, Part $2$ of Theorem $3)$\label{cor: p large}\\
Let $G$ be a finitely presented group with presentation $Q=\langle X|R\rangle$, such that the $p$-deficiency of $Q$ is one. Say $X$ is finite and $R=\{ w_1,\ldots,w_r,w_{r+1}^{p^{a_1}},\ldots, w_{r+l}^{p^{a_l}}\}$, where $l\geq 1$ and $a_j\geq 1$ for $1\leq j\leq l$. Suppose the order of $\psi(w_{r+j})$ in $G/R_G$ is strictly less than $p^{a_j}$, for some $j$, $1\leq j\leq l$. Then $G$ is $p$-large.
\end{cor}
\begin{proof}
Consider the presentation $P=\langle X|S\rangle$, where $S$ consists of all the elements of $R$ except $w_{r+j}^{p^{a_j}}$. Note that $def_p(P)>1$ and follow the same arguments as in \cref{thm: main section p large}.
\end{proof}
\vskip 2mm

\begin{ex}\label{ex: baumslag-solitar}
The group $B(m,n)$ given by the presentation
\[
\langle a,b \mid a^{-1}b^{m}a=b^{n}\rangle,
\]
where $m$ and $n$ are non-zero integers, is called the Baumslag-Solitar group of type $(m,n)$. These are known to be residually finite (see \cite{Meskin}) if and only if one of the three following conditions holds: $|m|=1$, $|n|=1$, or $|m|=|n|$. By \cite{moldavanskii}, the finite residual of a non-residually finite Baumslag-Solitar group, denoted by $R_{B(m,n)}$, is generated as a normal subgroup in $B(m,n)$, by the set of commutators of the form $[a^{k}b^{d}a^{-k},b]=a^{k}b^{d}a^{-k}ba^{k}b^{-d}a^{-k}b^{-1}$, where $k$ takes all possible integer values and $d$ is the greatest common divisor of $m$ and $n$. 
\vskip 2mm

Suppose $m$ and $n$ are such that $B(m,n)$ is not residually finite. Denote $[a^{k}b^{d}a^{-k},b]$ by $w_k$. By expressing $w_k$ in its reduced normal form (remember $B(m,n)$ is an HNN extension) and using Brittons Lemma (page $181$, \cite{lyndon}), then $w_k$ is non-trivial and has infinite order in $B(m,n)$. 

\vskip 2mm

Consider the free product $H=B(m,n)\ast\mathbb{Z}$. Subgroups of residually finite groups are residually finite. Since $B(m,n)$ is not residually finite, then $H$ is not residually finite. Moreover, as $B(m,n)\leqslant H$, then $R_{B(m,n)}\leqslant R_H$. Furthermore, as $R_H$ is normal in $H$, $t^{-1}R_{B(m,n)}t\leqslant R_H$, for all $t\in H$.
\vskip 2mm

Consider $u_k=t^{-1}w_kt$ for $k\in\mathbb{Z}$, where $t$ is the generator of $\mathbb{Z}$ in $H$. Consider the set $\{u_{i_l}  \}_{l\in I}$, where $I$ is a finite set of integers, and a collection of non-zero integers $\{n_{i_l}\}_{l\in I}$ with a common prime factor. Denote by $G$ the quotient of $H$ by $\langle\langle  u_{i_l}^{n_{i_l}} \rangle\rangle_{l\in I}$, the normal subgroup in $H$ generated by $\{u_{i_l}^{n_{i_{l}}}  \}_{l\in I}$. This quotient has a presentation
\[
Q=\langle a,b,t \mid a^{-1}b^{m}a=b^{n}, u_{i_{l}}^{n_{i_{l}}}=1  \rangle_{l\in I}.
\]
As $u_k\in R_H$ for all $k\in\mathbb{Z}$, then $R_G$ corresponds to $R_H$ under the canonical surjective homomorphism from $H$ to $G$. Moreover, $o(\psi(u_{i_{l}}), G/R_G)<p^{\nu_p(n_{i_{l}})}$, for all $l\in I$, where $\nu_p(n_{i_{l}})$ is the number of copies of $p$ that appear in the prime factorisation of $n_{i_{l}}$.
\vskip 2mm

Enough elements $u_{i_{l}}$ and suitable powers may be chosen, so that $Q$ has $p$-deficiency strictly less than one. However, as $o(\psi(u_{i_{l}}), G/R_G)<p^{\nu_p(u_{i_{l}})}$ for all $l\in I$ and $P=\langle a,b,t \mid a^{-1}b^{m}a=b^{n}\rangle$ has $p$-deficiency greater than one, \cref{thm: main section p large} says $G$ is $p$-large. 
\vskip 2mm

Note the construction of $G$ gives a method for constructing finite presentations of arbitrary large negative $p$-deficiency which are $p$-large. In particular, these groups are non-amenable and do not have property ($\tau$).
\vskip 2mm

\cref{thm: main of section} can be applied to the example described above to conclude $G$ has positive rank gradient even if $I$ is an infinite set of integers. However, the full force of \cref{thm: main of section} is not needed. The following argument suffices.
\vskip 2mm

Consider a finitely generated group $G$ and $N$ a normal subgroup of $G$ contained in $R_G$. Note that the finite index subgroups of $G/N$ are in bijection with the finite index subgroups of $G$. If $H$ is a finite index subgroup of $G$, then under this bijection, $H$ corresponds to $H/N$ in $G/N$. Moreover, as the commutator of any finite index subgroup of $G$ contains $R_G$, then the $p$-rank of $H$ is equal to the $p$-rank of $H/N$. Therefore, if $G$ has $p$-deficiency greater than one, then by \cref{eq: in first lemma} 
\[
0< def_p(G)-1\leq \dfrac{def_p(H)-1}{|G:H|}\leq \dfrac{d_p(H)-1}{|G:H|}
\]
\[
=\dfrac{d_p(H/N)-1}{|G:H|}\leq \dfrac{d(H/N)-1}{|G:H|},
\]
where $H$ is an arbitrary finite index normal subgroup of $G$. Then $RG(G/N)>0$ by taking the infimum over all finite index normal subgroups of $G$ on the right hand side of the inequality. A very similar argument, in the case when $R_G=N$, is used in \cite{puchta} to construct finitely generated infinite residually finite $p$-groups with positive rank gradient.
\vskip 2mm

Consider a finitely generated non-residually finite group with $RG(G)>0$, such as $H=B(m,n)\ast\mathbb{Z}$. By the argument above, any quotient of $H$ by a normal subgroup contained in the its finite residual gives a group with positive rank gradient. As $G$ is obtained by taking the quotient of $H$ by elements $u_{i_{l}}^{n_{i_{l}}}\in R_H$, then $G$ has positive rank gradient. In fact, this holds for any $n_{i_{l}}\in\mathbb{Z}$ and not just non-zero integers with a common prime factor.

\end{ex}

\vskip 2mm

\begin{ex}
The following presentation
\[
\langle a,b,s,t \ |\  [a,b]=1,a^s=(ab)^{2}, b^t=(ab)^2\rangle,
\]
where $a^s=s^{-1}as$, $b^t=t^{-1}bt$ and $[a,b]=aba^{-1}b^{-1}$, defines a group $G$ which is non-Hopfian \cite{wise}. Non-hopfian implies non-residually finite, therefore $G$ is non-residually finite.
\vskip 2mm

The author of \cite{wise} exhibited the following surjective endomorphism with non-trivial kernel. Consider $\psi:G\longrightarrow G$ defined by sending $t$ to $t$, $s$ to $s$, $a$ to $a^2$ and $b$ to $b^2$. As $\psi\big( (ab)^{s^{-1}} \big)=a$ and $\psi\big( (ab)^{t^{-1}} \big)=b$, $\psi$ is surjective. Moreover,
\[
e_G\neq[(ab)^{s^{-1}},(ab)^{t^{-1}}]\overset{\psi}{\longmapsto}[(a^2b^2)^{s^{-1}},(a^2b^2)^{t^{-1}}]=[a,b]=e_G.
\]
The fact that $e_G\neq[(ab)^{s^{-1}},(ab)^{t^{-1}}]$ may be checked by using Brittons Lemma (page $181$, \cite{lyndon}).
\vskip 2mm

Note that as $G/ker(\psi)$ is isomorphic to $G$, then the finite residual $R_G$ of $G$ contains $ker(\psi)$. Therefore, $[(ab)^{s^{-1}},(ab)^{t^{-1}}]$ is a non-trivial element of $G$ in $R_G$.
\vskip 2mm

Since $\psi\big( (ab)^{s^{-1}} \big)=a$ and $\psi\big( (ab)^{t^{-1}} \big)=b$, then $\psi\Big(\big( (ab)^{s^{-1}} (ab)^{t^{-1}}\big)^{s^{-1}}\Big)=(ab)^{s^{-1}}$ and $\psi\Big(\big( (ab)^{s^{-1}} (ab)^{t^{-1}}\big)^{t^{-1}}\Big)=(ab)^{t^{-1}}$. Define, 
\[
w_1:=[\big( (ab)^{s^{-1}} (ab)^{t^{-1}}\big)^{s^{-1}},\big( (ab)^{s^{-1}} (ab)^{t^{-1}}\big)^{t^{-1}}].
\]
As $\psi(w_1)=[(ab)^{s^{-1}},(ab)^{t^{-1}}]$ is non-trivial, then $w_1$ is non-trivial and $w_1\notin ker(\psi)$. However, as $[(ab)^{s^{-1}},(ab)^{t^{-1}}]\in ker(\psi)$, then $w_1\in ker(\psi^2)$. This means that $ker(\psi)$ is properly contained in $ker(\psi^2)$. 
\vskip 2mm

Denote $[(ab)^{s^{-1}},(ab)^{t^{-1}}]$ by $w_0$. So far we have defined $w_1$ from $w_0$ in such a way that $\psi(w_1)=w_0$. Moreover, $w_1\in ker(\psi^2)\backslash ker(\psi)$. By using induction, we now prove $ker(\psi^i)$ is properly contained in $ker(\psi^{i+1})$, for all $i\in\mathbb{N}$, by exhibiting an element $w_i\in ker(\psi^{i+1})\backslash ker(\psi^{i})$. 
\vskip 2mm

Suppose we have $w_{i-1}=[u_1,u_2]$ such that $w_{i-1}\in ker(\psi^{i})\backslash ker(\psi^{i-1})$. We define $w_i$ by $[(u_1u_2)^{s^{-1}},(u_1u_2)^{t^{-1}}]$. Assume $w_{i-1}$ was defined the same way using $w_{i-2}$. That is, $w_{i-1}=[(v_1v_2)^{s^{-1}},(v_1v_2)^{t{-1}}]$ where $w_{i-2}=[v_1,v_2]$. By the induction hypothesis, we know that $(v_1v_2)^{s^{-1}}=u_1$ goes to $v_1$ and $(v_1v_2)^{t{-1}}=u_2$ to $v_2$ under $\psi$. Therefore, $(u_1u_2)^{s^{-1}}$ goes to $(v_1v_2)^{s^{-1}}$ and $(u_1u_2)^{t^{-1}}$ to $(v_1v_2)^{t^{-1}}$. This means that $w_i$ goes to $w_{i-1}$ under $\psi$. Moreover, $w_{i-1}\in ker(\psi^{i})\backslash ker(\psi^{i-1})$. Hence, $w_i\in ker(\psi^{i+1})\backslash ker(\psi^{i})$.
\vskip 2mm

Since $\psi$ is a surjective endomorphism of $G$, then $\psi^{i}$ is too. This implies that $ker(\psi^i)\leqslant R_G$. Therefore, the elements $w_i$ are in $R_G$ and since $w_i\in ker(\psi^{i+1})\backslash ker(\psi^{i})$, for all $i\in \mathbb{N}$, then they are not conjugate to one another.
\vskip 2mm

Now construct a group $H$ from $G$ in the way it was done in \cref{ex: baumslag-solitar}: take $H=G\ast \mathbb{Z}$ which has deficiency greater than one. Let $z$ generate $\mathbb{Z}$ and consider $zw_iz^{-1}$, which we denote by $h_i$, for all $i\in\mathbb{N}$. As $R_G\leqslant R_H$ and $R_H$ is normal in $H$, then $h_i\in R_H$ for all $i\in\mathbb{N}$. Any group of the form 
\[
H/\langle\langle  h_i^{p^{a_i}}   \rangle\rangle_{i\in I},
\] 
where $a_i\geq 1$ and $I$ is a subset of $\mathbb{N}$, satisfies the hypotheses of \cref{thm: main of section}. If $I$ is finite it then satisfies the hypotheses of \cref{thm: main section p large}. In each case the group given by the presentation above has $p$-deficiency less than one but has positive rank gradient and is $p$-large (the latter condition if $I$ is finite).
\end{ex}

\section{Presentations with $p$-deficiency one}

Given an element $r$ in $F_n$, the non-abelian free group of rank $n$, we will consider its \emph{minimal root} in $F_n$, which we define as an element $u\in F_n$, such that $u^m=r$, where $m$ is the largest integer which can appear in an expression of the type $v^l=u$.
\vskip 2mm

We will also use the notion of \emph{minimal $p$-root}, which we now define. The author of \cite{puchta} defined $\nu_{p}(r)$, for $r$ in $F_n$, to be the supremum over all integers $a$, such that there exists some $w$ in $F_n$ with $r=w^{p^{a}}$. Call $w$ the \emph{minimal $p$-root} of $r$. 
\vskip 2mm

Let $r=u^m\in F_n$, where $u$ is the minimal root of $r$. Factorise $m$ as $p^{a}d$, where $a\geq 0$ and $(p,d)=1$. Then, $w=u^d$ is the minimal $p$-root of $r$.
\vskip 2mm

In this section we will use the notion of residual deficiency introduced in \cite{me} which we now present.
\vskip 2mm

\begin{defn}
Let $G$ be a finitely presented group with finite presentation $Q=\langle X|R\rangle$, 
where $X$ freely generates $F_n$, the non-abelian group of rank $n$. 
Let $R=\{u^{s_1}_1,\ldots,u^{s_m}_m\}$, where $u_i$ is the minimal root 
of $u_i^{s_i}$ for all $i$, $1\leq i\leq m$. Suppose the order of $\psi(u_i)$ in the 
residual quotient of $G$ is $k_i$, for all $i$. Then we define the 
\emph{residual deficiency} of the presentation $Q$ to be
\[
rdef(Q)=n-\sum_{i=1}^{m}\dfrac{1}{k_i}.
\]
We define the \emph{residual deficiency} of the group $G$ to be the
supremum  of the residual deficiencies defined by all finite presentations of $G$
\[
rdef(G)=\underset{\langle X|R\rangle \cong G}{\text{sup}} \left\{ rdef(Q) \right\}.
\]
\end{defn}

\vskip 1mm

In \cite{me} we proved that the following holds for some finite index subgroups of $G$
\[
rdef(G)-1\leq \dfrac{def(H)-1}{|G:H|}.
\]
\begin{rem}\label{rem: remark for thm}
This means that if the residual deficiency of $Q$ is greater than one, then the group has finite index subgroups with deficiency greater than one. This has strong consequences: having deficiency greater than one implies $p$-deficiency greater than one for every prime $p$ which implies $p$-largeness for every prime $p$ and positive rank gradient. Moreover, if the residual deficiency is equal to one, then there exists a finite index normal subgroup $H$ of $G$ with deficiency at least one. As the deficiency of $H$ is a lower bound for the rank of the abelianisation of $H$, then $H$ surjects onto $\mathbb{Z}$.
\end{rem}
\vskip 2mm

\begin{thm}\label{thm: main one}
Let $G$ be a finitely presented group and $Q=\langle X|R\rangle$ a finite presentation for
$G$ with $p$-deficiency one. Let $R=\{r_1,\ldots,r_d \}$ and let
$r_i=u_i^{m_i}=w_i^{p^{a_i}}$, for all $i$, $1\leq i\leq d$, where $u_i$ and
$w_i$ are the minimal and the $p$-minimal root of $r_i$,
respectively. Denote $o(\psi(u_i),G/R_G)$ by
$k_i$ and $o(\psi(w_i),G/R_G)$ by $l_i$ for all $i$,
$1\leq i\leq d$. Then,
\begin{enumerate}
\item If $l_i<p^{a_i}$ for some $i$, then $G$ is $p$-large and $RG(G)>0$. 
\item Suppose $l_i=p^{a_i}$ for all $i$. If $l_i<k_i$ for some $i$, then $G$ has a finite index subgroup which is $p$-large. Moreover, $G$ is large and $RG(G)>0$.
\end{enumerate}
\end{thm}
\begin{proof}
\begin{enumerate}
\item Follows from \cref{cor: rg>0 when order of minimal p root is less than its p power} and \cref{cor: p large}.
\item If $l_i=p^{a_i}$ for all $i$ and $l_i< k_i$ for some $i$, then the residual deficiency of $G$ is greater than one and by \cref{rem: remark for thm}, $G$ has the desired properties.

\end{enumerate}

\end{proof}
\vskip 2mm

\begin{rem}\label{rem: afer main thm}
Since $l_i\leq k_i$ and $l_i\leq p^{a_i}$, for all $i$, $1\leq i\leq d$, then, with the exception of when $k_i=l_i=p^{a_i}$ for all $i$, all possible relationships between $l_i$, $k_i$ and $p^{a_i}$ are considered in \cref{thm: main one}. When $k_i=l_i=p^{a_i}$ for all $i$, the residual deficiency is equal to one and hence, by \cref{rem: remark for thm}, $G$ has a finite index subgroup which surjects onto $\mathbb{Z}$.
\end{rem}
\vskip 2mm

\begin{rem}
The author of \cite{anitha} proves that if def$_p(Q)=1$, then $G$ has a finite index subgroup that surjects onto $\mathbb{Z}$. The proof is split into two cases. The first considers the situation when $l_i=k_i=p^{a_i}$. Here, the author of \cite{anitha} makes use of a result in \cite{allcock} to conclude $G$ has a finite index subgroup that surjects onto $\mathbb{Z}$. When $l_i<p^{a_i}$, the author of \cite{anitha} proves \cref{cor: p large} (\cite{anitha}, Theorem $3$ part $2$) to conclude $G$ is $p$-large. 
\end{rem}
\vskip 2mm

\begin{cor}
Let $G$ be a finitely presented group with a finite presentation $Q$ such that def$_p(Q)=1$. Then $G$ does not have property $(\tau)$. In particular it does not have property $(T)$. Moreover, $G$ is non-amenable with the possible exception of when $k_i=l_i=p^{a_i}$.

\end{cor}
\begin{proof}
By \cref{thm: main one} the group $G$ is $p$-large unless $k_i=l_i=p^{a_i}$. As $p$-largeness implies non-amenability and not having property ($\tau$), the result follows in this case.
\vskip 2mm

If $k_i=l_i=p^{a_i}$, then by \cref{rem: afer main thm} we have that $G$ has a finite index subgroup $H$ that surjects onto $\mathbb{Z}$. As property ($\tau$) is invariant under finite index subgroups and quotients, and as $\mathbb{Z}$ does not have property ($\tau$), then $G$ does not have property ($\tau$). 
\end{proof}

\section{Examples}

\begin{ex}

Consider the generalised triangle group given by the presentation
\[
\langle a,b \mid a^3,b^3,w^{3n}\rangle,
\]
where
\[ 
w = a_1^{r_1} b_1^{s_1}\cdots a_k^{r_k} b_k^{s_k} \ \ \ \ \ (k\geq 1,1\leq r_i < 3,1 \leq s_i < 3), \ \ \ n>1.
\]
The $3$-deficiency of this presentation is one and its $p$-deficiency is less than one for any other prime. However, by \cite{gen traingle groups} these groups are residually finite and the order of the elements $a$, $b$ and $w$ in the corresponding generalised triangle group is the one given in the presentation, hence the residual deficiency of the group is $2-2/3-1/3n$, which is greater than one if $n>1$. Moreover, if $n>1$, $o(\psi(w),G/R_G)>3$, which then implies $G$ has a finite index normal subgroup $H$ with def$_p(H)>1$.
\vskip 2mm

A similar thing holds for other well known families of residually finite groups such as the Coxeter groups. 
\vskip 2mm

Remember that a Coxeter group is given by a presentation of the following type
\[
\langle a_1,\ldots,a_n  \mid a_1^{2},\ldots,a_n^{2},(a_ia_j)^{m_{ij}},1\leq i<j\leq n \rangle,
\]
where $m_{ij}\geq 2$. The value $m_{ij}$ may be $\infty$ in which case the word $(a_ia_j)^{m_{ij}}$ is omitted from the set of relators. Coxeter groups are known to be residually finite and the order of the elements $a_1,\ldots,a_n$ and $a_ia_j$, in the group defined by the presentation above, is the one specified in the presentation. Just as with the generalised triangle groups, it is easy to find examples of presentations for Coxeter groups which have $2$-deficiency less than or equal to one, but which have residual deficiency greater than one and have finite index subgroups with $p$-deficiency greater than one.
\end{ex}

\vskip 2mm

\begin{ex}
Consider the group $G$ given by the presentation 
\[
Q=\langle x_1,x_2,t \mid v_1^{p},\ldots, v_{2p-1}^{p},w^{pq}\rangle,
\]
where $p$ and $q$ are distinct prime numbers. Suppose that for all $i$, $1\leq i\leq 2p-1$, $v_i$ and $w$ cannot be expressed as a proper power of some other element in $F_3$. Denote by
\[
\varphi: F_3\longrightarrow C_p \times C_p\times C_q
\]
the map defined by sending the generators $x_1$, $x_2$ and $t$, to $(1,0,0)$, $(0,1,0)$ and $(0,0,1)$, respectively. Moreover, denote by
\[
\psi_1: F_3\longrightarrow C_p\times C_p, \ \ \ \ \ \ 
\psi_2: F_3\longrightarrow C_q,
\]
the maps defined by composing $\varphi$ with the projection onto the first two factors, in the case of $\psi_1$, and $\varphi$ composed with the projection onto the third factor, in the case of $\psi_2$.
\vskip 2mm

Assume $\sigma_t(v_i)$, the power sum of $t$ in $v_i$, is zero for all $i$, $1\leq i\leq 2p-1$, and $\sigma_t(w)\not\equiv 0$ (mod $q$). Also assume $\psi_1(v_i)\neq 0$ and $\psi_1(w)\neq 0$, for all $i$. Note that the map $\varphi$ lifts to a map $\overline{\varphi}$ from $G$ to $C_p\times C_p\times C_q$, and the order of $\overline{\varphi}(v_i)$ in $C_p\times C_p\times C_q$, for all $i$, is $p$, while the order of $\overline{\varphi}(w)$ in $C_p\times C_p\times C_q$ is $pq$. Therefore, the $q$-deficiency of $Q$ is less than one, the $p$-deficiency is one, but the residual deficiency is greater than or equal to $3-(2p-1)/p-1/pq$ which is greater than one.
\end{ex}
\vskip 2mm

\begin{ex}
Consider the group $G$ constructed in \cref{ex: baumslag-solitar}
\[
G\cong \big ( B(m,n)\ast \mathbb{Z}\big )/\langle\langle u_{i}^{n_{i}}   \rangle\rangle_{i\in I}.
\]
Suppose $I=\{1,\ldots,p\}$, $n_{i}=p$ for all $i\in I$, and that for $1\leq i\leq p$, $u_i$ is not a proper power of any other element. Hence, $G$ admits a presentation
\[
Q=\langle a,b,t \mid b^{-1}a^{m}b=b^{n}, u_{1}^{p}=\cdots =u_p^{p}=1  \rangle,
\]
which has $p$-deficiency one.
\vskip 2mm

Given that $u_{1}^{p},\ldots,u_p^{p}\in R_{H}$, where $H=B(m,n)\ast \mathbb{Z}$, then $o(\psi(u_i),G/R_G)=1<p$ for all $i$, $1\leq i\leq p$, and hence by \cref{thm: main one}
part $2$, $G$ is $p$-large and $RG(G)>0$. 
\end{ex}
\vskip 2mm

As observed in \cref{rem: afer main thm}, if $k_i=l_i=p^{a_i}$ for all $i$, $1\leq i\leq d$, then there is a finite index normal subgroup $H$ in $G$ which surjects onto $\mathbb{Z}$. Within the class of groups that satisfy these conditions, there are examples of groups which are $p$-large  and examples of groups which are not $p$-large, just as there are examples of groups which have positive rank gradient and others that do not. Therefore, the results presented so far are insufficient for concluding or discarding stronger properties when the finitely presented group has $p$-deficiency one and $k_i=l_i=p^{a_i}$.
\vskip 2mm

\begin{ex}
Consider a Coxeter group $C$, given by the presentation
\[
\langle a_1,\ldots,a_n  \mid a_1^{2},\ldots,a_n^{2},(a_ia_j)^{m_{ij}},1\leq i<j\leq n \rangle,
\]
where all the labels $m_{ij}$ are a power of a prime $p$ (or $\infty$). The $p$-Coxeter subgroup of $C$, as defined in \cite{jack-anitha}, is the index two subgroup given by the kernel of $\theta: C\longrightarrow C_2$, where $\theta(x_i)=1$, for all $i$, $1\leq i\leq n$. In \cite{jack-anitha}, a presentation for the $p$-Coxeter subgroup of $C$ is computed using the presentation for $C$ and the Reidemeister-Schreier rewriting process. The presentation for the $p$-Coxeter subgroup thus obtained, is given by
\[
\langle x_1,\ldots,x_{n-1}  \mid x_1^{m_{12}},\ldots,x_{n-1}^{m_{1n}},(x_{i-1}x_{j-1})^{m_{ij}},2\leq i<j\leq n \rangle.
\]
Denote by $S_n(p)$ the $p$-Coxeter subgroup of $C$. As $S_n(p)$ is a subgroup of $C$ and $C$ is residually finite, then $S_n(p)$ is residually finite. Moreover, the words $x_1,\ldots,x_{n-1}$ and $x_{i-1}x_{j-1}$, for $2\leq i<j\leq n $, are non trivial in $S_n(p)$ and hence its residual deficiency is
\[
n-1-\sum^{n}_{j=2}\dfrac{1}{m_{1j}}-\sum_{2\leq i<j\leq n}\dfrac{1}{m_{ij}},
\]
where $1/m_{ij}=0$ if $m_{ij}=\infty$.
\vskip 2mm

In particular, consider $C$ the Coxeter group given by
\[
\langle a_1,\ldots,a_4  \mid a_1^{2},\ldots,a_4^{2},(a_ia_j)^{3} \rangle,\ \ \ \ 1\leq i<j\leq 4,
\]
and its $3$-Coxeter subgroup $S_4(3)$ given by
\[
\langle x_1,x_2,x_{3}  \mid x_1^{3},x_2^{3},x_3^{3},(x_{1}x_{2})^{3}, (x_{1}x_{3})^{3}, (x_{2}^{-1}x_{3})^{3} \rangle.
\]
The $3$-deficiency and residual deficiency of the previous presentation are both one, and $k_i=l_i=3$, for $i=1,\ldots,6$. However, this group was proved to be $3$-large in \cite{jack-anitha}. To our knowledge, it is not known whether $RG(S_4(3))>0$ or $RG(S_4(3))=0$. 
\vskip 2mm

On the other hand, the generalised triangle group given by
\[
\langle a_1,a_2 \mid a_1^{3}, a_2^{3}, (a_1a_2)^{3}\rangle,
\]
has $3$-deficiency and residual deficiency equal to one, but has a subgroup of index three which is isomorphic to $\mathbb{Z}\times \mathbb{Z}$. Therefore, this group is not large and does not have positive rank gradient. 

\end{ex}

\end{document}